\documentclass[12pt]{article}
\usepackage{amssymb}
\usepackage[mathscr]{eucal}
\usepackage{amsthm,amsmath,anysize}
\def\a{\alpha}

\def\l{\lambda}
\def\g{\gamma}
\def\0{\bar{0}}
\def\1{\bar{1}}

\def\g{\mathfrak{g}}

\def\g{{\mathfrak g}}

\newtheorem{lemma}{Lemma}[section]
\newtheorem{theorem}[lemma]{Theorem}

\newtheorem{proposition}[lemma]{Proposition}

\hoffset \voffset \oddsidemargin=60pt \evensidemargin=45pt
\title{ On the simplicity of  induced modules for  reductive Lie algebras}
\author{
Chaowen Zhang\\ Department of
Mathematics,\\ China University
 of Mining and Technology,\\ Xuzhou, 221116, Jiang Su, P. R. China}

\date{ }

\begin{document}
\maketitle

\section{Introduction}
  Let $G$ be a reductive algebraic group defined over an algebraically closed field $\mathbf F$ of positive characteristic $p$, and let $\g$ be the Lie algebra of $G$.  In \cite[5.1]{fp2}, Friedlander and Parshall  raised an open question (stated below) about the simplicity of  certain induced $\g$-modules with $p$-character $\chi\in \g^*$. The question has been answered (by V. Kac) when $\g$ is of type $A_2$ (see \cite[Example 3.6]{fp2}), and also when $\g$ is of type $A_3$ (see \cite{lsy}).  In this paper, we study the question under certain restriction on $\chi$.  \par
   Following \cite[6.3]{j} we make the following hypotheses:\par
   (H1) The derived group $DG$ of $G$ is simply connected;\par
   (H2) The prime $p$ is good for $\g$;\par
   (H3) There exists a $G$-invariant non-degenerate bilinear form on $\g$.\par
   Let $T$ be a maximal torus of $G$, let $\mathfrak h=\mbox{Lie}(T)$, and let $\Phi$ be the root system of $G$.  Let $\Pi=\{\a_1, \dots,\a_l\}$ be a base of $\Phi$ and let  $\Phi^+$ be the set of positive    roots relative to $\Pi$. For each $\a\in\Phi^+$ let $\g_{\a}$ denote the corresponding root space of $\g$.   According to \cite[6.1]{j} we have  $\g=\mathfrak n^-+\mathfrak h+\mathfrak n^+$, where $$\mathfrak n^+=\sum_{\a\in\Phi^+}\g_{\a}, \quad \mathfrak n^-=\sum_{\a\in\Phi^+}\g_{-\a}.$$    Fix a  proper subset $I$ of $\Pi$ and put $\Phi_I=\mathbb ZI\cap\Phi$ and $ \Phi^+_I=\Phi_I\cap \Phi^+.$  Define   $\tilde\g_I=\mathfrak h+\sum_{\a\in\Phi_{I}}\g_{\a}$, as well as
$$\mathfrak u=\sum_{\a\in \Phi^+\setminus \Phi^+_I}\g_{\a},\quad \mathfrak u'=\sum_{\a\in \Phi^+\setminus \Phi^+_I}\g_{-\a}.$$ Then $\mathfrak p_I=\tilde\g_I+\mathfrak u$ and $\mathfrak p'_I=\tilde\g_I+\mathfrak u'$ are parabolic subalgebras of $\g$, each with Levi factor $\tilde\g_I$ \cite[10.6]{j}.
  Throughout the paper we assume  that $\chi (\mathfrak n^+)=0$. This is done without loss of generality due to \cite[Lemma 6.6]{j}. Our method requires the additional assumption that $\chi (\mathfrak u')=0$, which we make throughout,\par For any restricted Lie subalgebra $L$ of $\g$,  we denote by $u_{\chi}(L)$ the {\it $\chi$-reduced enveloping algebra} of $L$, where we continue to use $\chi$ for the restriction of $\chi$ to $L$ (\cite[5.3]{sf}).  If $\chi =0$,  $u_{\chi}(L)$ is referred to as the {\it reduced enveloping algebra} of $L$, and denoted more simply  by $u(L)$. Let $L^{\chi}_I(\lambda)$ be a simple $u_{\chi}(\mathfrak p_I)$-module generated by a maximal vector $v_{\lambda}$ of weight $\lambda\in \mathfrak h^*$. Define the induced $u_{\chi}(\g)$-module $$Z^{\chi}_I(\lambda)=u_{\chi}(\g)\otimes _{u_{\chi}(\mathfrak p_I)}L^{\chi}_I(\lambda).$$ \par
     The main result of the present paper is Th. 3.7, which  gives a  necessary and sufficient condition for   $Z^{\chi}_I(\lambda)$ to be simple;  we show that $Z^{\chi}_I(\lambda)$ is simple if and only if $\lambda$ is not a zero of a certain polynomial $R^I_{\g}(\lambda)$. Under our assumption on $\chi$,  Theorem 3.7 answers the open question \cite[5.1]{fp2}.  \par  The paper is organized as follows. In Sec. 2,  we introduce the concept {\it extended $\a$-string} for any simple root $\a$ in an irreducible root system; we also investigate extended $\a$-strings in all different irreducible root systems (see Prop. 2.1).  Using  results from Sec. 2, we prove the main theorem in Sec. 3, which says that the simplicity of the induced module $Z^{\chi}_I(\lambda)$ is completely determined by a polynomial $R^I_{\g}(\lambda)$. In Sec. 4 we establish the explicit expression of the polynomial $R^I_{\g}(\lambda)$; we also show that Th. 4.4 recovers the Kac-Weisfeiler theorem (\cite{fp1,kw}).\par

\section{$\a$-strings in a root system}
Let $\Pi$ and $\Phi^+$ be as above. Without loss of generality we assume $\Phi$ is irreducible. For each  $\a\in \Pi$ and $\beta\in \Phi^+\setminus \a$,
we denote the $\a$-string through $\beta$ by $\mathcal S_{\a}\beta$. Define an order on the set $\mathcal S_{\a}\beta$ by $$\beta+q\a\prec \beta+(q-1)\a\prec \cdots \prec \beta\prec \beta-\a\prec \cdots\prec \beta-r\a,$$
where $q$ (resp. $r$) is the largest non-negative integer such that $\beta+q\a$ (resp. $\beta-r\a$) in $\Phi^+$. By \cite[9.4]{hu}, the length of the string is at most 4. We say that the $\a$-string through $\beta$ is {\it isolated} if $r=q=0$. Note that if $\mathcal S_{\a}\beta$ is non-isolated, we have $\mathcal S_{\a}\beta=\mathcal S_{\a}\beta'$ for any $\beta'\in \mathcal S_{\a}\beta$. To avoid repetitions, we assume in the following  that $\beta+\a\notin \Phi^+$.\par
We call the set $((\mathbb N\setminus 0)\beta+\mathbb Z\a)\cap \Phi^+$ the {\it extended $\a$-string through $\beta$}, denoted $\tilde{\mathcal S}_{\a}\beta$. Define an order on the extended $\a$-string by $$l\beta+m\a\prec l'\beta+m'\a \quad\text{if $l>l'$ or $l=l'$ but $m>m'$}.$$

\begin{proposition}Assume that $\Phi$ is irreducible and  not of type $G_2$. Let $\a\in\Pi$, and let $\beta\in \Phi^+\setminus \a$ with $\mathcal S_{\a}\beta$ non-isolated. Then we have either $\mathcal S_{\a}\beta=\{\beta,\beta-a\}$ or $\mathcal S_{\a}\beta=\{\beta,\beta-\a,\beta-2\a\}$,  and either $\tilde{\mathcal S}_{\a}\beta= \mathcal S_{\a}\beta$ or $\tilde{\mathcal S}_{\a}\beta=\{2\beta-\a\}\cup\mathcal S_{\a}\beta$.\par
\end{proposition}\begin{proof} Set $$\Phi_{\a,\beta}=(\mathbb Z\a+\mathbb Z\beta)\cap \Phi,\quad \Phi_{\a,\beta}^+=(\mathbb Z\a+\mathbb Z\beta)\cap \Phi^+,\quad \Phi_{\a,\beta}^-=(\mathbb Z\a+\mathbb Z\beta)\cap \Phi^-.$$ Then clearly $\Phi_{\a,\beta}=\Phi^+_{\a,\beta}\cup\Phi_{\a,\beta}^-$ is
a subsystem of rank 2. In addition, $\a\in\Phi^+_{\a,\beta}$ is also a simple root. By assumption, the subsystem $\Phi_{\a,\beta}$ can only be  of  type  $ A_2$ or $ B_2$.\par If $\Phi_{\a,\beta}$ is of  type $A_2$, then we have $\Phi_{\a,\beta}^+=\{\a, \beta, \beta-\a\}$, so that $$\tilde{\mathcal S}_{\a}\beta=\mathcal S_{\a}\beta=\{\beta,\beta-\a\}.$$ If $\Phi_{\a,\beta}$ is of  type $B_2$, then we have $$\Phi^+_{\a,\beta}=\{\a_1,\a_2,\a_1+\a_2,\a_1+2\a_2\}$$ with either $\a=\a_1$ or $\a=\a_2$. Since $\mathcal S_{\a}\beta$ is non-isolated, in the case  $\a=\a_1$, we must have $\beta=\a_1+\a_2$. It follows that $$\mathcal S_{\a}\beta=\{\beta,\beta-\a\},\quad  \tilde{\mathcal S}_{\a}\beta=\{2\beta-\a\}\cup\mathcal S_{\a}\beta.$$ In case $\a=\a_2$, we must have $\beta=\a_1+2\a_2$, so that $$\tilde{\mathcal S}_{\a}\beta=\mathcal S_{\a}\beta=\{\beta,\beta-\a, \beta-2\a\}.$$
\end{proof}
We now discuss the case $G_2$.  According to \cite[Ch. 6, 4.13]{bo}, we have $$\Phi^+=\{\a_1,\a_2, \a_1+\a_2, 2\a_1+\a_2,3\a_1+\a_2, 3\a_1+2\a_2\}, \quad \Pi=\{\a_1,\a_2\}.$$ Case 1. $\a=\a_1$.
For  $\beta=3\a_1+2\a_2$, the $\a$-string $\mathcal S_{\a}\beta$ is isolated; for $\beta=3\a_1+\a_2$, we have $$\mathcal S_{\a}\beta=\{3\a_1+\a_2, 2\a_1+\a_2, \a_1+\a_2,\a_2\}\quad\mathbin{\mathrm{and}}\quad
\tilde{\mathcal S}_{\a}\beta=\Phi^+\setminus \a.$$
Case 2. $\a=\a_2$. For $\beta_1=3\a_1+2\a_2$, we have $$\tilde {\mathcal S}_{\a}\beta_1=\mathcal S_{\a}\beta_1=\{3\a_1+2\a_2, 3\a_1+\a_2\};$$ for $\beta_2=2\a_1+\a_2$, the $\a$-string through it is isolated;
for $\beta_3=\a_1+\a_2$, we have $\mathcal S_{\a}\beta_3=\{\a_1+\a_2,\a_1\}$  and    $$\tilde{\mathcal S}_{\a}\beta_3=\{3\a_1+2\a_2, 3\a_1+\a_2, 2\a_1+\a_2,\a_1+\a_2,\a_1\}=\Phi^+\setminus \a.$$ Note that $\tilde{\mathcal S}_{\a}\beta_1\subseteq \tilde{\mathcal S}_{\a}\beta_3$. \par
Let $\Phi$ be irreducible and let $\a\in \Pi$.  If $\beta_1, \beta_2\in \Phi^+\setminus\a$ with $\mathcal S_{\a}\beta_1$ and $\mathcal S_{\a}\beta_2$ both non-isolated,  we have from Prop. 2.1 that  $\tilde{\mathcal S}_{\a}\beta_1=\tilde{\mathcal S}_{\a}\beta_2$ or $\tilde{\mathcal S}_{\a}\beta_1\cap \tilde{\mathcal S}_{\a}\beta_2=\phi$ if $\Phi$ is not of type $G_2$, but we can have $\tilde{\mathcal S}_{\a}\beta_1\subsetneqq\tilde{\mathcal S}_{\a}\beta_2$ in the case  $\Phi$ is of type $G_2$.
\section{Simplicity criterion}In this section, we keep the assumptions as in the introduction.  Let $$\{e_{\a}, h_{\beta}|\a\in\Phi, \beta\in\Pi\}$$ be a  Chevalley basis for $\g'=\text{Lie}(G')$ such that $$[e_{\a}, e_{\beta}]=\pm (r+1)e_{\a+\beta},\quad \mathbin{\mathrm{if}}\quad\a,\ \beta, \ \a+\beta\in \Phi^+,$$ where $r$ is the greatest integer for which $\beta-r\a\in\Phi$ (see \cite[Th. 25.2]{hu}). From the proof of Prop. 2.1, we see that our assumption on $p$ ensures that $(r+1)\neq 0$.\par  For $\a\in \Phi^+$ put $f_{\a}=-e_{\a}$. Then we have $\g_{\a}=\mathbf Fe_{\a}$ and $\g_{-\a}=\mathbf Ff_{\a}$ for every $\a\in\Phi^+$.   For each fixed simple root $\a$, let \ $N_{\a}=\sum_{\beta\in \Phi^+\setminus\a}\g_{-\a}$. \ Then $N_{\a}$ is a restricted subalgebra of $\g$. \par

Let $u(N_{\a})$ be the restricted enveloping algebra of $N_{\a}$. For each $\beta\in\Phi^+\setminus\a$ with $\mathcal S_{\a}\beta$ non-isolated, we define $\tilde f^{\beta}\in u(N_{\a})$ to be the product of $f_{\gamma}^{p-1}$, $\gamma\in \tilde{\mathcal S}_{\a}\beta$  in the order given in Sec. 2. For example, if \ $\tilde{\mathcal S}_{\a}\beta=\mathcal S_{\a}\beta=\{\beta,\beta-\a\}$, then  $$\tilde f^{\beta}=f_{\beta}^{p-1}f_{\beta-\a}^{p-1}\in u(N_{\a}).$$ \par

 Remark: Let $\a,\beta\in\Phi^+$ such that $\a+\beta\in\Phi^+$ (resp. $\beta-\a\in\Phi^+$). Then we have $$[e_{\a}, e_{\beta}]=ce_{\a+\beta}, \quad [f_{\a}, f_{\beta}]=-cf_{\a+\beta} \ (\text{resp. $[e_{\a}, f_{\beta}]=cf_{\beta-\a}$})$$ for some $c\in\mathbf F\setminus 0$. For  brevity, we omit the scalar $c$. This does not affect any of the proofs in this section.
\begin{lemma} Let $\a\in\Pi$. For each $\beta\in\Phi^+\setminus \a$ with $\mathcal S_{\a}\beta$ non-isolated, we have $$ [e_{\a}, \tilde{f}^{\beta}]=0.$$
\end{lemma} \begin{proof} We may assume that $\Phi$ is irreducible. Suppose that $\Phi$ is not of type $G_2$. By Prop. 2.1, we need only consider the following cases. \par
 Case 1. $\tilde{\mathcal S}_{\a}\beta=\mathcal S_{\a}\beta=\{\beta, \beta-\a\}$. Then we have $\tilde f^{\beta}=f_{\beta}^{p-1}f_{\beta-\a}^{p-1}$. Since $$[e_{\a}, f_{\beta}]=f_{\beta-\a},\quad [e_{\a},f_{\beta-\a}]=0, \quad [f_{\beta-\a},f_{\beta}]=0,$$ and  $f_{\beta-\a}^p=0$ in $u(N_{\a})$, the lemma follows.\par Case 2.  $\tilde{\mathcal S}_{\a}\beta=\{2\beta-\a,\beta,\beta-\a\}$, $\mathcal S_{\a}\beta=\{\beta, \beta-\a\}$.  In this case we have $$\tilde f^{\beta}=f_{2\beta-\a}^{p-1}f_{\beta}^{p-1}f_{\beta-\a}^{p-1}.$$ Since $$[e_{\a}, f_{2\beta-\a}]=0,\quad [e_{\a},f_{\beta}]=f_{\beta-\a}, \quad [f_{\beta-\a},f_{\beta}]=f_{2\beta-\a},$$ and $[f_{\beta},f_{2\beta-\a}]=0$, we get $[e_{\a}, \tilde f^{\beta}]=0$.\par
 Case 3. $\tilde{\mathcal S}_{\a}\beta=\mathcal S_{\a}\beta=\{\beta, \beta-\a, \beta-2\a\}$. In this case  we have $$\tilde f^{\beta}=f_{\beta}^{p-1}f_{\beta-\a}^{p-1}f_{\beta-2\a}^{p-1}.$$ Since $$[e_{\a}, f_{\beta}]=f_{\beta-\a}, \quad [e_{\a}, f_{\beta-\a}]=f_{\beta-2\a}, \quad [e_{\a}, f_{\beta-2\a}]=0, \quad [f_{\beta-\a}, f_{\beta}]=0,$$ and
$[f_{\beta-2\a}, f_{\beta-\a}]=0$, we have  $[e_{\a}, \tilde f^{\beta}]=0$.\par
Assume $\Phi$ is of type $G_2$. In the case  $\a=\a_1$, $\beta=3\a_1+\a_2$,  we have  from Sec. 2 that $$\tilde f^{\beta}=f_{3\a_1+2\a_2}^{p-1}f_{3\a_1+\a_2}^{p-1}f_{2\a_1+\a_2}^{p-1}f_{\a_1+\a_2}^{p-1}f_{\a_2}^{p-1}.$$ Since $[e_{\a_1}, f_{3\a_1+2\a_2}]=0$, we have $$e_{\a_1}\tilde f^{\beta}=f_{3\a_1+2\a_2}^{p-1}e_{\a_1}f_{3\a_1+\a_2}^{p-1}f_{2\a_1+\a_2}^{p-1}f_{\a_1+\a_2}^{p-1}f_{\a_2}^{p-1}$$
$$(\text{using $[e_{\a_1}, f_{3\a_1+\a_2}]=f_{2\a_1+\a_2}$ and $[f_{2\a_1+\a_2}, f_{3\a_1+\a_2}]=0$})$$
$$=f_{3\a_1+2\a_2}^{p-1}f_{3\a_1+\a_2}^{p-1}e_{\a_1}f_{2\a_1+\a_2}^{p-1}f_{\a_1+\a_2}^{p-1}f_{\a_2}^{p-1}$$$$(\text{using $[e_{\a_1}, f_{2\a_1+\a_2}]=f_{\a_1+\a_2}$, $[f_{\a_1+\a_2}, f_{2\a_1+\a_2}]=f_{3\a_1+2\a_2}$,}$$$$\text{ and the fact that $f_{3\a_1+2\a_2}$ commutes with all $f_{\beta}$, $\beta\in \Phi^+$})$$$$=f_{3\a_1+2\a_2}^{p-1}f_{3\a_1+\a_2}^{p-1}f_{2\a_1+\a_2}^{p-1}e_{\a_1}f_{\a_1+\a_2}^{p-1}f_{\a_2}^{p-1}$$
$$(\text{using $[e_{\a_1}, f_{\a_1+\a_2}]=f_{\a_2}$ and $[f_{\a_1+\a_2}, f_{\a_2}]=0$})$$$$=\tilde f^{\beta}e_{\a_1},$$  so that $[e_{\a}, \tilde f^{\beta}]=0.$\par Let $\a=\a_2$. For $\beta_1=3\a_1+2\a_2$, we have from Sec. 2 that $\tilde f^{\beta_1}= f_{\beta_1}^{p-1}f_{\beta_1-\a}^{p-1}$, and hence $[e_{\a}, \tilde f^{\beta_1}]=0$ as above.  For $\beta_3=\a_1+\a_2$, we have  $$\tilde f^{\beta_3}=f_{3\a_1+2\a_2}^{p-1}f_{3\a_1+\a_2}^{p-1}f_{2\a_1+\a_2}^{p-1}f_{\a_1+\a_2}^{p-1}f_{\a_1}^{p-1}.$$ Since \ $[e_{\a_2}, f_{3\a_1+\a_2}]=0$ and $[e_{\a_2}, f_{2\a_1+\a_2}]=0$, \ it is easy to see that \ $[e_{\a}, \tilde f^{\beta_3}]=0$.

\end{proof}
Recall from the introduction  the notation $\mathfrak p_I, \mathfrak p'_I, \mathfrak u, \mathfrak u'$, and $\tilde\g_I$.
Each simple $u_{\chi}(\mathfrak p_I)$-module is generated by a maximal vector $v_{\lambda}$ of weight $\lambda\in \mathfrak h^*$, denoted $L^{\chi}_{I}(\lambda)$. Define the induced $u_{\chi}(\g)$-module $$Z_I^{\chi}(\lambda)=u_{\chi}(\g)\otimes_{u_{\chi}(\mathfrak p_I)}L^{\chi}_{I}(\lambda).$$ By the PBW theorem for the $\chi$-reduced enveloping algebra $u_{\chi}(\g)$ (\cite[Th. 5.3.1]{sf}), we have $$Z_I^{\chi}(\lambda)\cong u_{\chi}(\mathfrak u')\otimes_{\mathbf F}L^{\chi}_{I}(\lambda)$$ as $u_{\chi}(\mathfrak u')$-modules. By the assumption on $\chi$, we have  $u_{\chi}(\mathfrak u')=u(\mathfrak u')$.\par
    Let $\Phi^+\setminus \Phi^+_I=\{ \beta_1,\beta_2,\dots, \beta_k\}$, and let $v_1,\dots, v_n$ be a basis of $L^{\chi}_I(\lambda)$.  Then $Z^{\chi}_I(\lambda)$ has a basis $$f_{\beta_1}^{l_1}f_{\beta_2}^{l_2}\cdots f_{\beta_k}^{l_k}\otimes v_j,\quad 0\leq l_i\leq p-1,\ i=1,\dots,k,\ j=1,\dots, n.$$
  Using (H3), we can show that $\mathfrak u$ is the nilradical of the parabolic subalgebra $\mathfrak p_I$.
By \cite[Coro. 3.8]{sf},  $L^{\chi}_I(\lambda)$ is  annihilated by $\mathfrak u$, and hence is a simple $u_{\chi}(\tilde\g_I)$-module.

\begin{lemma} For any fixed ordering of $\Phi^+\setminus \Phi^+_I$: $\beta_{i_1},\dots,\beta_{i_k}$, there is a nonzero scalar $c\in\mathbf F$ such that $f_{\beta_{i_1}}^{p-1}\cdots f_{\beta_{i_k}}^{p-1}=cf_{\beta_1}^{p-1}\cdots f_{\beta_k}^{p-1}$ in $u(\mathfrak u')$.
\end{lemma}
\begin{proof}Since $u(\mathfrak u')$ is restricted, it is naturally a $T$-module under the adjoint representation. There is a PBW type  basis for $u(\mathfrak u')$ given as (\cite[p. 1057]{fp1}): $$f_{\beta_1}^{l_1}\cdots f_{\beta_{i_k}}^{l_k},\quad 0\leq l_1,\dots,l_k\leq p-1.$$ The $T$-weight of each element  $f_{\beta_1}^{l_1}\cdots f_{\beta_{i_k}}^{l_k}$ is exactly $\sum^k_{s=1}l_s\beta_{i_s}$. Write $f_{\beta_1}^{p-1}\cdots f_{\beta_k}^{p-1}$ as a linear combination of the above basis: $$f_{\beta_1}^{p-1}\cdots f_{\beta_k}^{p-1}=\sum c_l f_{\beta_{i_1}}^{l_1}\cdots  f_{\beta_{i_k}}^{l_k}.$$ By comparing the $T$-weights we see that all coefficients $c_l$ must be zero except for the one for $f_{\beta_{i_1}}^{p-1}\cdots f_{\beta_{i_k}}^{p-1}$, which is nonzero, since $f_{\beta_1}^{p-1}\cdots f_{\beta_k}^{p-1}$ is an element of another basis for $u(\mathfrak u')$. This completes the proof.
\end{proof}

\begin{lemma} Let $\a\in I$, and let $\beta\in\Phi^+\setminus \Phi^+_I$. If  $\mathcal S_{\a}\beta$ is non-isolated,  then  $$\tilde{\mathcal S}_{\a}\beta\subseteq \Phi^+\setminus \Phi^+_I.$$
\end{lemma} \begin{proof} Since $\mathfrak u$ is an ideal of $\mathfrak p_I$ with roots $\Phi^+\setminus\Phi^+_I$,  we have  $\mathcal S_{\a}\beta\subseteq \Phi^+\setminus \Phi^+_I$. \par
 Suppose that $\Phi$ is irreducible and  not of type $G_2$.  By Prop. 2.1 we have $$\tilde{\mathcal S}_{\a}\beta=\mathcal S_{\a}\beta\quad \mbox{or}\quad \tilde{\mathcal S}_{\a}\beta=\{2\beta-\a\}\cup\mathcal S_{\a}\beta.$$
 The statement clearly holds in the case $\tilde{\mathcal S}_{\a}\beta=\mathcal S_{\a}\beta$. So we  assume  $$\tilde{\mathcal S}_{\a}\beta=\{2\beta-\a\}\cup\mathcal S_{\a}\beta.$$  Since
$\mathcal S_{\a}\beta\subseteq \Phi^+\setminus \Phi^+_I$, we have  $\beta-\a$, $\beta\in \Phi^+\setminus \Phi^+_I$; that is, $\beta-\a$ and $\beta$ are roots of $\mathfrak u$. It follows that
 $e_{\beta-\a}, e_{\beta}\in \mathfrak u$, and hence, $e_{2\beta-\a}\in \mathfrak u$. Therefore, $2\beta-\a$ is also a root of $\mathfrak u$, implying $\tilde{\mathcal S}_{\a}\beta\subseteq \Phi^+\setminus\Phi^+_I$.\par Suppose $\Phi$ is of type $G_2$. For $I=\{\a_1\}$, let $\a=\a_1$ and $\beta=3\a_1+\a_2$. Then we have by the discussion in Sec.2 that $$\mathcal S_{\a}\beta=\{\beta, \beta-\a,\beta-2\a, \beta-3\a\},\quad \tilde{\mathcal S}_{\a}\beta=\{2\beta-3\a\}\cup \mathcal S_{\a}\beta.$$ Assume $I=\{\a_2\}$.
 For $\a=\a_2$ and $\beta_1=3\a_1+2\a_2$, we have $$\tilde{\mathcal S}_{\a}\beta_1=\mathcal S_{\a}\beta_1=\{\beta_1, \beta_1-\a\}.$$
 For $\a=\a_2$ and $\beta_3=\a_1+\a_2$, we have from Sec. 2 that $$\tilde{\mathcal S}_{\a}\beta_3=\{3\beta_3-\a, 3\beta_3-2\a, 2\beta_3-\a, \beta_3, \beta_3-\a\},\quad \mathcal S_{\a}\beta_3=\{\beta_3, \beta_3-\a\}.$$ In each case above we have $\tilde{\mathcal S}_{\a}\beta\subseteq \Phi^+\setminus \Phi^+_I$. \par Suppose $\Phi$ is a disjoint union of irreducible subsystems. Then $I$ is a disjoint union of the  subsets of simple roots  in these subsystems.  Let $\a \in I$ and let $\beta\in \Phi^+\setminus \Phi^+_I$ with $\mathcal S_{\a}\beta$ non-isolated. Then $\a, \beta$ are  in the same irreducible subsystem.  Thus,  we have by the above discussion that $\tilde{\mathcal S}_{\a}\beta\subseteq \Phi^+\setminus\Phi^+_I$.\end{proof}
Let $\a\in I$. By the lemma, we see that the set $\Phi^+\setminus \Phi^+_I$ is a disjoint union of all different  $\tilde{\mathcal  S}_{\a}\beta$ with non-isolated $\mathcal S_{\a}\beta$ and the isolated $\mathcal S_{\a}\beta=\{\beta\}$.
 We order the elements in $\Phi^+\setminus\Phi^+_I$ in such a way that all elements in the same $\tilde{\mathcal S}_{\a}\beta$ with $\mathcal S_{\a}\beta$ non-isolated  are adjacent in the order defined in Sec. 2, and call it an {\it $\a$-order}.\par
 Let $\mathfrak S$ be a subset of $\Phi^+$. We say that $\mathfrak S$ is  {\it a closed subset} if $\a+\beta\in \mathfrak S$ for any $\a, \beta\in \mathfrak S$ such that $\a+\beta\in \Phi^+$. Therefore, $\Phi^+\setminus \Phi^+_I$ is a closed subset of $\Phi^+$. We see that $\mathfrak S$ is a closed subset of $\Phi^+$ if and only if \ $\mathfrak s=:\sum_{\a\in\mathfrak S}\g_{-\a}$ \ is a Lie subalgebra of $\g$;  it is clear that $\mathfrak s$ is restricted.\par Let $\mathfrak S$ be a closed  subset of $\Phi^+$. \ Applying almost verbatim Humphreys's argument
 in the proof of \cite[Lemma 1.4]{hu1}, we get the following result.
 \begin{lemma} Let $(\a_1,\dots, \a_m)$ be any ordering of $\mathfrak S$. If $ht(\a_k)=h$, assume that all exponents $i_j$ in \ $f_{\a_1}^{i_1}\cdots f_{\a_m}^{i_m}\in u(\mathfrak s)$ \ for which $ht(\a_j)\geq h$ are equal to $p-1$. Then, if $f_{\a_k}$ is inserted anywhere into this expression, the result is $0$.
 \end{lemma}

\begin{lemma} Let $\Phi^+\setminus\Phi^+_I=\{\beta_1, \dots, \beta_k\}$. For $f_{\beta_1}^{p-1}\cdots f_{\beta_k}^{p-1}\in u(\mathfrak u')$, we have in $u_{\chi}(\g)$ that
  $$[e_{\a}, f_{\beta_1}^{p-1}\cdots f_{\beta_k}^{p-1}]=0,\quad  [f_{\a}, f_{\beta_1}^{p-1}\cdots f_{\beta_k}^{p-1}]=0$$ for every $\a\in \Phi^+_I$.
\end{lemma}\begin{proof} By the remark before Lemma 3.1 it suffices to prove the identities for $\a\in I$.\par  For each $\a\in I$, we put the set $\Phi^+\setminus \Phi^+_I$ in a fixed $\a$-order:  $$\beta_{i_1}\prec \cdots\prec \beta_{i_k}.$$  Then $f_{\beta_{i_1}}^{p-1}\cdots f_{\beta_{i_k}}^{p-1}$ is the product of $\tilde f^{\beta}$ for non-isolated $\mathcal S_{\a}\beta$ and $f_{\beta}^{p-1}$ with $\mathcal S_{\a}\beta$ isolated.
By Lemma 3.1, $e_{\a}$ commutes with  every $\tilde f^{\beta}$ with $\mathcal S_{\a}\beta$ non-isolated.  It is clear that $[e_{\a}, f_{\beta}^{p-1}]=0$ if $\mathcal S_{\a}\beta$ is isolated.  Then we have $[e_{\a}, f_{\beta_{i_1}}^{p-1}\cdots f_{\beta_{i_k}}^{p-1}]=0$, and hence $[e_{\a}, f_{\beta_1}^{p-1}\cdots f_{\beta_k}^{p-1}]=0$ by Lemma 3.2.\par To prove the second identity, we apply Lemma 3.4. Recall that $\mathfrak u$ is an ideal of $\mathfrak p_I$ with roots $\Phi^+\setminus \Phi^+_I$. Then  for each $\beta_i\in\Phi^+\setminus \Phi^+_I$, we have \ $\a+\beta_i\in \Phi^+\setminus \Phi^+_I$ if $\a+\beta_i\in\Phi^+$. If $[f_{\a}, f_{\beta_i}]=0$ for all $i$, then it is trivially true that $$[f_{\a}, f_{\beta_1}^{p-1}\cdots f_{\beta_k}^{p-1}]=0.$$ Assume that $[f_{\a}, f_{\beta_i}]\neq 0$ for some $i$. Then  we have $[f_{\a}, f_{\beta_i}]=f_{\a+\beta_i}$. Since $ht(\a+\beta_i)>ht(\beta_i)$ for all $\beta_i$ with $\a+\beta_i\in\Phi^+$, we have by Lemma 3.4 that $$[f_{\a}, f_{\beta_1}^{p-1}\cdots f_{\beta_k}^{p-1}]=\sum _{i, \a+\beta_i\in\Phi^+}\sum^{p-2}_{s=0}f_{\beta_1}^{p-1}\cdots (f_{\beta_i}^sf_{\a+\beta_i}f_{\beta_i}^{p-2-s})\cdots f_{\beta_k}^{p-1}=0.$$
\end{proof}
\begin{lemma} There is a uniquely determined scalar $R^I_{\g}(\lambda)\in \mathbf F$ such that
$$e_{\beta_1}^{p-1}\cdots e_{\beta_k}^{p-1}f_{\beta_1}^{p-1}\cdots f_{\beta_k}^{p-1}\otimes v_{\l}=R^I_{\g}(\l)\otimes v_{\lambda}$$ in $Z^{\chi}_I(\lambda)$.
\end{lemma}
\begin{proof}
Let $U(\g)$ (resp. $U(\mathfrak h)$) be the universal enveloping algebra of $\g$ (resp. $\mathfrak h$). Fix an ordering $\a_1,\dots, \a_t$ of the positive roots $\Phi^+$. By the PBW theorem for $U(\g)$, we have $$e_{\beta_1}^{p-1}\cdots e_{\beta_k}^{p-1}f_{\beta_1}^{p-1}\cdots f_{\beta_k}^{p-1}=f(h)+\sum u^-_iu^0_iu^+_i,$$ with $f(h), u^0_i\in U(\mathfrak h)$ and where each $u^+_i$ (resp. $u^-_i$)  is of the form $$e_{\a_1}^{l_1}\cdots e_{\a_t}^{l_t}(\text{resp.}\quad f_{\a_1}^{s_1}\cdots f_{\a_t}^{s_t}), \quad l_j, s_j\geq 0,$$ with $u^+_i$ and $u^-_i$ not both equal to $1$.  Note that $U(\g)$ is  naturally a $T$-module under the adjoint representation.  Let us denote the $T$-weight of  a weight vector $u\in U(\g)$ by $\mbox{wt}(u)$.  Since $$\mbox{wt}(e_{\beta_1}^{p-1}\cdots e_{\beta_k}^{p-1}f_{\beta_1}^{p-1}\cdots f_{\beta_k}^{p-1})=0,$$   we have $\mbox{wt}(u^+_i)=-\mbox{wt}(u^-_i)\neq 0$ for every $i$.  It follows that $\sum^t_{i=1}l_i>0$, for every $u^+_i=e_{\a_1}^{l_1}\cdots e_{\a_t}^{l_t}$.  \par We use for the images of the generators $e_{\a}, f_{\a}, h_{\a}$ in $u_{\chi}(\g)$  the same notation as before in $U(\g)$. By our assumption we have $$u_{\chi}(\g)=u_{\chi}(\mathfrak n^-)u_{\chi}(\mathfrak h)u(\mathfrak n^+).$$ Then we have in $u_{\chi}(\g)$:$$(*)\quad e_{\beta_1}^{p-1}\cdots e_{\beta_k}^{p-1}f_{\beta_1}^{p-1}\cdots f_{\beta_k}^{p-1}=\bar f(h)+\sum \bar u^-_i\bar u^0_i\bar u^+_i,$$  where $\bar f(h), \bar u^0_i\in u_{\chi}(\mathfrak h), \bar u^{-}_i\in u_{\chi}(\mathfrak n^{-})$, $\bar u^+_i\in u(\mathfrak n^+)$.\par  For each $u^+_i=e_{\a_1}^{l_1}\cdots e_{\a_t}^{l_t}\in U(\g)$, if $l_s\geq p$ for some $s$, then $\bar u^+_i=0$. On the other hand,  if $l_s\leq p-1$ for every $s$, the $\bar u^+_i=e_{\a_1}^{l_1}\cdots e_{\a_t}^{l_t}\in u_{\chi}(\g)$, so $\bar u_i^+\neq 1$ (since $u^+_i\neq 1$). It follows that $\bar u_i^+v_{\l}=0$.  Applying both sides of  $(*)$ to  $1\otimes v_{\lambda}$, we have in $Z^{\chi}_I(\lambda)$ that
$$e_{\beta_1}^{p-1}\cdots e_{\beta_k}^{p-1}f_{\beta_1}^{p-1}\cdots f_{\beta_k}^{p-1}\otimes v_{\lambda}=1\otimes \bar f(h)v_{\lambda}=R^I_{\g}(\lambda)\otimes v_{\lambda}$$ for some scalar $R^I_{\g}(\lambda)$.
\end{proof}
\begin{theorem} The $u_{\chi}(\g)$-module $Z^{\chi}_I(\lambda)$ is simple if and only if $R^I_{\g}(\lambda)\neq 0$.
\end{theorem}
\begin{proof} Suppose $R^I_{\g}(\lambda)\neq 0$. Using the PBW theorem for the $\chi$-reduced enveloping algebra $u_{\chi}(\g)$ (\cite[Theorem 5.3.1]{sf}) and our assumption that $\chi (\mathfrak u')=0$, we have a  natural vector space isomorphism $$Z^{\chi}_I(\l)\cong u_{\chi}(\mathfrak u')\otimes _{\mathbf F}L^{\chi}_I(\l)=u(\mathfrak u')\otimes _{\mathbf F}L^{\chi}_I(\l).$$ Put the elements in  $\Phi^+\setminus \Phi^+_I$  in the order of ascending heights:  $\beta_1,\dots, \beta_k$.
Therefore, $Z^{\chi}_I(\l)$ has as basis the set $\{f_{\beta_1}^{l_1}\cdots f_{\beta_k}^{l_k}\otimes v_j|0\leq l_i\leq p-1, 1\leq j\leq n\}$, where $\{v_j|1\leq j\leq n\}$ is a basis for $L^{\chi}_I(\l)$.\par Let $N$ be a nonzero submodule of $Z_I^{\chi}(\lambda)$.  There exists a nonzero element $x\in N$, which we can write $$x=\sum_l c_l f_{\beta_1}^{l_1}\cdots f_{\beta_k}^{l_k}\otimes v_l, $$ where the sum is over all tuples $l=(l_1,\dots,l_k)$ with $
0\leq l_i\leq p-1$ and where $c_l\in\mathbf F$ and $v_l\in L^{\chi}_I(\l)$. By applying appropriate $f_{\beta_i}$'s, we get $f_{\beta_1}^{p-1}\cdots f_{\beta_k}^{p-1}\otimes v\in N$  for some nonzero $v\in L^{\chi}_I(\lambda)$.\par
 It follows from hypothesis (H3) in the introduction that $\mathfrak u$ is the nilradical of the parabolic subalgebra $\mathfrak p_I$. By \cite[Coro. 3.8]{sf}, $L^{\chi}_I(\l)$ is annihilated by $\mathfrak u$, and is hence a simple $u_{\chi}(\tilde\g_I)$-module.  Therefore $u_{\chi}(\tilde\g_I)v=L^{\chi}_I(\lambda)$. Then using Lemma 3.5, we have $$\begin{aligned} f_{\beta_1}^{p-1}\cdots f_{\beta_k}^{p-1}\otimes L^{\chi}_I(\lambda)&=f_{\beta_1}^{p-1}\cdots f_{\beta_k}^{p-1}\otimes u_{\chi}(\tilde\g_I)v\\& \subseteq u_{\chi}(\tilde\g_I)f_{\beta_1}^{p-1}\cdots f_{\beta_k}^{p-1}\otimes v\\ &\subseteq N,\end{aligned}$$ so that $f_{\beta_1}^{p-1}\cdots f_{\beta_k}^{p-1}\otimes v_{\lambda}\in N$. By Lemma 3.6, $$R^I_{\g}(\lambda)\otimes v_{\lambda}=e_{\beta_1}^{p-1}\cdots e_{\beta_k}^{p-1}f_{\beta_1}^{p-1}\cdots f_{\beta_k}^{p-1}\otimes v_{\lambda}\in N,$$ and hence $1\otimes v_{\lambda}\in N$, implying  $N=Z_I^{\chi}(\lambda)$. We conclude  that $Z_I^{\chi}(\lambda)$ is simple.\par
Suppose that $Z^{\chi}_I(\lambda)$ is simple. Recall the definition of the parabolic subalgebra $\mathfrak p'_I=\tilde\g_I+\mathfrak u'$. Since $f_{\beta_i}f_{\beta_1}^{p-1}\cdots f_{\beta_k}^{p-1}=0$ for all $i$, it follows that  $$L^{\chi}_I(\l)'=:f_{\beta_1}^{p-1}\cdots f_{\beta_k}^{p-1}\otimes L^{\chi}_I(\lambda)$$ is a $u_{\chi}(\mathfrak p'_I)$-module that is isomorphic to $L^{\chi}_I(\lambda)$ as vector spaces. The canonical $u_{\chi}(\g)$-module homomorphism  $$\varphi: u_{\chi}(\g)\otimes_{u_{\chi}(\mathfrak p'_I)} L^{\chi}_I(\lambda)'\longrightarrow Z^{\chi}_I(\lambda)$$ induced by the embedding $L^{\chi}_I(\lambda)'\subseteq Z^{\chi}_I(\l)$ is trivially nonzero, and is therefore surjective since $Z^{\chi}_I(\lambda)$ is simple. Comparing the dimensions we see that $\varphi$ must be an isomorphism.\par  Now $v_{\l}$ is nonzero, so $v=:e_{\beta_1}^{p-1}\cdots e_{\beta_k}^{p-1}\otimes (f_{\beta_1}^{p-1}\cdots f_{\beta_k}^{p-1}\otimes v_{\lambda})$ is a nonzero element of  $u_{\chi}(\g)\otimes_{u_{\chi}(\mathfrak p'_I)} L^{\chi}_I(\lambda)'$. Therefore, $$R^I_{\g}(\lambda)\otimes v_{\l}=e_{\beta_1}^{p-1}\cdots e_{\beta_k}^{p-1}f_{\beta_1}^{p-1}\cdots f_{\beta_k}^{p-1}\otimes v_{\lambda}=\varphi (v)\neq 0,$$ implying $R^I_{\g}(\l)\neq 0$.
\end{proof}
Let us look at an application of Theorem 3.7. In \cite[5.1]{fp2}, Friedlander and Parshall asked the following question: Can one give necessary and sufficient condition on a simple module for a parabolic subalgebra $\mathfrak p_I$ to remain simple upon induction to $\g$.  Clearly under our assumption the question is answered by the theorem.\par
\section{A formula for $R^I_{\g}(\lambda)$}
In this section,  we  determine  $R^I_{\g}(\lambda)$ using the polynomial defined by Rudakov (\cite{ru}).  Recall the notation $\g'$ and  $\tilde\g_I$. Define $\g_I=[\tilde\g_I, \tilde\g_I]$.  Since $\g'\supseteq [\g,\g]\supseteq \g_I$ by \cite[Coro.10.5]{hu1}, $\g_I$ is spanned by a subset of the Chevalley basis of $\g'$. This ensures the application of \cite[Prop. 8]{ru} to $\g_I$.
For each $\a\in\Phi$, we shall write $\a$ instead of its derivative $d\a$ by abuse of notation.  \par Let $\chi\in \g^*$ as given earlier.
 Then $\chi$ can be written as $\chi=\chi_s+\chi_n$, with $\chi_s(\mathfrak n^++\mathfrak n^-)=0$ and $\chi_n(\mathfrak h+\mathfrak n^+)=0$.
 For each simple $u_{\chi_s}(\mathfrak p_I)$-module $L^{\chi_s}_I(\lambda)$ ($\lambda\in \mathfrak h^*$), define the induced module $$Z^{\chi_s}_I(\lambda)=u_{\chi_s}(\g)\otimes_{u_{\chi_s}(\mathfrak p_I)}L^{\chi_s}_I(\lambda).$$ Let $v_{\l}\in L^{\chi_s}_I(\lambda)$ be a maximal vector of weight $\lambda$. In a similar way as in the last section we define the scalar $R^I_{\g}(\lambda)_s$ by $$e_{\beta_1}^{p-1}\cdots e_{\beta_k}^{p-1}f_{\beta_1}^{p-1}\cdots f^{p-1}_{\beta_k}\otimes v_{\lambda}=R^I_{\g}(\lambda)_s\otimes v_{\lambda}.$$
\begin{lemma} $R^I_{\g}(\lambda)_s=R^I_{\g}(\lambda)$ for any $\lambda\in \mathfrak h^*$.
\end{lemma}
\begin{proof}From the last section, we have in $U(\g)$ that $$(1)\quad e_{\beta_1}^{p-1}\cdots e_{\beta_k}^{p-1}f_{\beta_1}^{p-1}\cdots f^{p-1}_{\beta_k}=f(h)+\sum u^-_iu^0_iu^+_i,$$ where each $u^+_i$ (resp. $u^-_i$) is in the form $e_{\a_1}^{l_1}\cdots e_{\a_t}^{l_t}$ (resp. $f_{\a_1}^{k_1}\cdots f_{\a_t}^{k_t}$) with $$l_1,\cdots, l_t, k_1,\cdots,k_t\in \mathbb N,\quad \sum^t_{i=1} l_i>0,\quad \sum^t_{i=1}k_i>0.$$  In view of the PBW type bases for $u_{\chi}(\g)$ and $u_{\chi_s}(\g)$ (see \cite[Th. 5.3.1]{sf}), we have the isomorphisms of vector spaces $$u_{\chi}(\g)\cong u_{\chi_n}(\mathfrak n^-)\otimes  u_{\chi_s}(\mathfrak h)\otimes  u(\mathfrak n^+),\quad \quad u_{\chi_s}(\g)\cong u(\mathfrak n^-) \otimes u_{\chi_s}(\mathfrak h)\otimes  u(\mathfrak n^+).$$  Then the images of the elements $f(h), u^0_i, u^+_i$ in (1) are the same in both $u_{\chi}(\g)$ and $u_{\chi_s}(\g)$. Applying the images of $(1)$ to $1\otimes v_{\lambda}$ in $Z^{\chi}_I(\lambda)$ and $Z^{\chi_s}_I(\lambda)$ respectively,  we obtain the same element $1\otimes \bar f(h)v_{\lambda}\in 1\otimes u_{\chi_s}(\mathfrak h)v_{\lambda}$. It follows that $R^I_{\g}(\lambda)_s=R^I_{\g}(\lambda)$.
\end{proof}
By the lemma, in calculating $R^I_{\g}(\lambda)$, we may assume $\chi=\chi_s$.  With this assumption, Lemma 3.2 says that any two products $e_{\a_1}^{p-1}\cdots e_{\a_t}^{p-1}\in u(\mathfrak n^+)$ (or $f_{\a_1}^{p-1}\cdots f_{\a_t}^{p-1}\in u(\mathfrak n^-$)) in different orders  are equal (up to scalar multiple).\par
 For the Borel subalgebra $\mathfrak b=\mathfrak h+\mathfrak n^+$ of $\g$,   let $\mathbf Fv_{\l}$ be the 1-dimensional $u_{\chi_s}(\mathfrak b)$-module with $v_{\lambda}$ a maximal vector of weight $\lambda\in \mathfrak h^*$. Define the induced $u_{\chi_s}(\g)$-module $$Z^{\chi_s}(\lambda)=u_{\chi_s}(\g)\otimes _{u_{\chi_s}(\mathfrak b)}\mathbf Fv_{\lambda}.$$
Put all positive roots in $\Phi^+$ in the order of ascending heights: $$\a_{i_1},\a_{i_2},\cdots, \a_{i_t}.$$ Let $h_{\a}=[e_{\a}, f_{\a}]$ for all $\a\in\Phi^+$.  Then we have by \cite[Prop. 8]{ru} that $$e_{\a_{i_1}}^{p-1}\cdots e_{\a_{i_t}}^{p-1}f_{\a_{i_1}}^{p-1}\cdots f_{\a_{i_t}}^{p-1}\otimes v_{\lambda}=R_{\g}(\lambda)\otimes v_{\lambda},$$ where $R_{\g}(\lambda)=(-1)^t\Pi^t_{i=1}[(\lambda+\rho)(h_{\a_i})^{p-1}-1].$\par
Let $\mathfrak b_I=\mathfrak b\cap \g_I$. Then $\mathfrak b_I$ is a Borel subalgebra of $\g_I$. Define the
induce $u_{\chi_s}(\g_I)$-module $u_{\chi_s}(\g_I)\otimes_{u_{\chi_s}(\mathfrak b_I)}\mathbf Fv_{\lambda}$, which can be canonically imbedded in $Z^{\chi_s}(\lambda)$. Put the roots in $\Phi^+_I$ in the order of ascending heights: $\a_{j_1}, \dots, \a_{j_s}$. Using \cite[Prop. 8]{ru} for $\g_I$,   we have $$e_{\a_{j_1}}^{p-1}\cdots e_{\a_{j_s}}^{p-1}f_{\a_{j_1}}^{p-1}\cdots f_{\a_{j_s}}^{p-1}\otimes v_{\lambda}=(-1)^s\Pi^s_{i=1}[(\lambda+\rho_I)(h_{\a_{j_i}})^{p-1}-1]\otimes v_{\lambda}$$ in $Z^{\chi_s}(\lambda)$, where $\rho_I=\frac{1}{2}\sum_{\a\in\Phi_I^+}\a$. We denote $ (-1)^s\Pi^s_{i=1}[(\lambda+\rho_I)(h_{\a_{j_i}})^{p-1}-1]$ by $R_{\g_I}(\lambda)$.\par
 Remark: As $\l\in\mathfrak h^*$ varies, each $\l (h_{\a})$ with $\a\in\Phi^+$ can be viewed as a (linear) polynomial on $\mathfrak h^*$ as follows:  For the basis $h_{\a_1},\dots, h_{\a_l}$ of $\mathfrak h$, let $h_{\a_1}^*,\dots, h_{\a_l}^*$ be a basis of $\mathfrak h^*$ such that $$h_{\a_i}^*(h_{\a_j})=\delta_{ij}\quad\mathbin{\mathrm{for}}\quad i,j=1,\dots,l.$$ Then  each $\l\in\mathfrak h^*$ can be written as $\l=\sum_{i=1}^lx_i h_{\a_i}^*$, $x_i\in\mathbf F$, so that $\l(h_{\a_i})=x_i$ for $i=1,\dots, l$. For each $\a\in\Phi^+$, using the property of the Chevalley basis (\cite[Th. 25.2(c)]{hu}) that $h_{\a}$ is a $\mathbb Z$-linear combination of $h_{\a_1},\dots, h_{\a_l}$, say $h_{\a}=\sum^l_{i=1} k_ih_{\a_i}$, we get $\l(h_{\a})=\sum^l_{i=1}k_ix_i$.  Therefore, $R_{\g}^I(\lambda)$, $R_{\g}(\l)$, and $R_{\g_I}(\l)$ are all polynomials in variables $x_1,\dots, x_l$.\par
\begin{lemma} $$R^I_{\g}(\lambda)R_{\g_I}(\lambda)=cR_{\g}(\lambda),\quad c\in\mathbf F\setminus 0.$$
  \end{lemma}
\begin{proof} Put the elements in $\Phi^+$ in the order  $\a_1,\dots, \a_t$ such  that $\a_{t-s+1}, \dots, \a_t$ are  positive roots  of $\g_I$ in the order of ascending heights, so that $$\Phi^+\setminus \Phi_I^+=\{\a_1,\dots,\a_{t-s}\}\ (\mbox{denoted $\{\beta_1,\dots,\beta_k\}$ earlier}).$$ By Lemma 3.2 and  analogous conclusions for $u(\mathfrak u), \ u(\mathfrak n^+)$, and $u(\mathfrak n^-)$, there is a nonzero $c\in\mathbf F$ such that $$cR_{\g}(\lambda)\otimes v_{\lambda}=e_{\a_1}^{p-1}\cdots e_{\a_{t-s}}^{p-1}e_{\a_{t-s+1}}^{p-1}\cdots e_{\a_t}^{p-1}f_{\a_1}^{p-1}\cdots f_{\a_{t-s}}^{p-1}f_{\a_{t-s+1}}^{p-1}\cdots f_{\a_t}^{p-1}\otimes v_{\lambda}.$$ By Lemma 3.5,  each $e_{\a_i}$, $t-s<i\leq t$, commutes with $f_{\a_1}^{p-1}\cdots f_{\a_{t-s}}^{p-1}$, so we get $$\begin{aligned} cR_{\g}(\lambda)\otimes v_{\lambda}&=e_{\a_1}^{p-1}\cdots e_{\a_{t-s}}^{p-1}f_{\a_1}^{p-1}\cdots f_{\a_{t-s}}^{p-1}(e_{\a_{t-s+1}}^{p-1}\cdots e_{\a_t}^{p-1}f_{\a_{t-s+1}}^{p-1}\cdots f_{\a_t}^{p-1}\otimes v_{\l})\\&=e_{\a_1}^{p-1}\cdots e_{\a_{t-s}}^{p-1}f_{\a_1}^{p-1}\cdots f_{\a_{t-s}}^{p-1} R_{\g_I}(\lambda)\otimes v_{\lambda}\\&=R_{\g_I}(\lambda)R^I_{\g}(\lambda)\otimes v_{\lambda}.\end{aligned}$$ This completes the proof.
\end{proof}
 To prove the next theorem, we need to apply (H3). Let $( , )$ be the nondegenerate bilinear form on $\g$. Define the mapping $\theta: \g\longrightarrow \g^*$ by $\theta (x)= (-, x)$ for all $x\in \g$. Let us note that $\g$ (resp. $\g^*$) is naturally a $G$-module with the adjoint (resp. coadjoint) action. Then the $G$-invariance of $( , )$ implies that $\theta$ is an  isomorphism of $G$-modules, so that $\theta$ is also an isomorphism of $\g$-modules by  \cite[7.11(3)]{j1}. Here $\g$ is a (left) $\g$-module with the $\g$-action given by $$ad x(y)=[x,y]\quad \mbox{for} \quad x,y\in \g,$$ whereas the $\g$-action on $\g^*$, by \cite[7.11(8)]{j1}, is defined by $$(x\cdot f)(y)=f(ad(-x)(y))\quad \mbox{for}\quad x,y\in\g,\ f\in\g^*.$$ Since $\theta$ is a $\g$-module isomorphism, it follows that $$\begin{aligned} (-, [x, y])&=\theta (adx (y))\\&=(-x)\cdot\theta (y)\\(\text{using the definition of $\g$-action on $\g^*$})&=([-,x], y)\end{aligned}$$ for all $x,y\in\g$; that is,  $( , )$ is also $\g$-invariant.\par
According to \cite[6.6]{j}, the bilinear form on $\g$ is also non-degenerate on $\mathfrak h$.
   For each $\lambda\in \mathfrak h^*$, let $t_{\lambda}\in \mathfrak h$ be  such that \ $\lambda (h)=(h, t_{\lambda}) \quad \text{for all}\quad h\in \mathfrak h.$
 Define the bilinear form $( , )$  on $\mathfrak h^*$ by $$(\lambda, \mu)=(t_{\lambda}, t_{\mu}),\quad \lambda,\mu\in\mathfrak h^*.$$ \begin{lemma} Let $W$ be the Weyl group of $G$ and let $w\in W$. Then $$(w\lambda, w\mu)=(\lambda, \mu)\quad \text{for}\quad  \lambda,\mu\in \mathfrak h^*.$$
 \end{lemma} \begin{proof} Let $g\in N_G(T)$ represent $w$. Then since $$(h, g^{-1}t_{g\lambda})=(gh, t_{g\lambda})=\lambda (h)=(h, t_{\lambda})$$ for all $h\in \mathfrak h$, so that $g^{-1}t_{g\lambda}=t_{\lambda}$, it follows that, for $\lambda, \mu\in \mathfrak h^*$,  $$\begin{aligned}(g\lambda, g\mu)&=(t_{g\lambda}, t_{g\mu})\\&=(g\mu)(t_{g\lambda})\\&=\mu(g^{-1}t_{g\lambda})\\&=\mu (t_{\lambda})\\&=(\lambda,\mu).\end{aligned}$$\end{proof}
  Keep the ordering of the elements of $\Phi^+$ as in the proof of Lemma 4.2. Then we have the following theorem.
  \begin{theorem} $$R^I_{\g}(\lambda)=c\Pi^{t-s}_{i=1}[(\lambda+\rho)(h_{\a_i})^{p-1}-1]$$ for some nonzero $c\in\mathbf F$.
  \end{theorem}
  \begin{proof} From above we have $$R_{\g}(\lambda)=(-1)^t\Pi^t_{i=1}[(\lambda+\rho)(h_{\a_i})^{p-1}-1]$$ and $$R_{\g_I}(\lambda)=(-1)^{s}\Pi^t_{i=t-s+1}[(\lambda+\rho_I)(h_{\a_i})^{p-1}-1].$$ Since $R_{\g}(\lambda$), $R^I_{\g}(\lambda)$,  and $R_{\g_I}(\lambda)$ are all elements in the polynomial algebra $\mathbf F[x_1,\dots,x_l]$,
  which contains no zero divisors, by the cancelation law and Lemma 4.2 it suffices to show that $\rho(h_{\a})=\rho_I(h_{\a})$ for all $\a\in\Phi^+_I=\{\a_{t-s+1},\dots, \a_t\}$.\par  For every $\a\in\Phi^+_I$, applying the argument for the proof \cite[Prop. 8.3(c)]{hu} we have, for all $h\in\mathfrak h$,
   $$\begin{aligned}(h, h_{\a})&=(h, [e_{\a}, f_{\a}])\\&=([h, e_{\a}], f_{\a})\\&=\a(h)(e_{\a}, f_{\a})\\&=(h, t_{\a})(e_{\a}, f_{\a})\\&=(h, (e_{\a}, f_{\a})t_{\a}),\end{aligned}$$ so that
   $h_{\a}=c_{\a}t_{\a}$, in which $c_{\a}=:(e_{\a}, f_{\a})$ is nonzero since the bilinear form is nondegenerate.  \par   If $\a\in I$, then we have  $$\begin{aligned}(\rho-\rho_I)(h_{\a})&=c_{\a}(\rho-\rho_I) (t_{\a})\\&=c_{\a}(t_{\a}, t_{\rho-\rho_I})\\&=c_{\a}(\a, \rho-\rho_I)\\(\mathbin{\mathrm{using \ Lemma \ 4.3}})&=c_{\a}(s_{\a}(\a), s_{\a}(\rho-\rho_I))\\&=
  c_{\a}(-\a, \rho-\rho_I)\\&=-(\rho-\rho_I)(h_{\a}),\end{aligned}$$ implying  that $\rho(h_{\a})=\rho_I(h_{\a})$.  For every  $\a\in\Phi^+_I$, by the property of the Chevalley basis mentioned before,  $h_{\a}$ is a $\mathbb Z$-linear combination of $h_{\a_i}, \a_i\in I$, so we have $\rho(h_{\a})=\rho_I(h_{\a})$.   This completes the proof.
  \end{proof}
  As an application of Theorem 4.4, we give a new proof of the Kac-Weisfeiler theorem (cf. \cite[Th. 8.5]{fp1}).
  \begin{theorem} Let $\g=\mbox{Lie}(G)$ be a restricted Lie algebra of classical type. Keep the assumptions from the introduction. Assume that $\chi(h_{\a})\neq 0$ for all $\a\in\Phi^+\setminus \Phi^+_I$. Then the induced module
  $Z^{\chi}_I(\lambda)$ is simple.
  \end{theorem}
  \begin{proof} Recall from the proof of Lemma 4.2 that $\Phi^+\setminus \Phi^+_I=\{\a_1,\dots, \a_{t-s}\}$. Since $\chi(h_{\a})\neq 0$ for all $\a\in\Phi^+\setminus \Phi^+_I$,  we have $$(\lambda+\rho)(h_{\a_i})^{p-1}-1\neq 0\quad \mathbin{\mathrm{for}}\quad i=1,\dots, t-s,$$ so that $R^I_{\g}(\lambda)\neq 0$. Thus, $Z^{\chi}_I(\lambda)$ is simple.

  \end{proof}

Acknowledgement: The author thanks an anonymous expert for pointing out the simple proof for Prop. 2.1. The author  thanks also Randall Holmes and the referee   for useful suggestions.

\def\refname{
\end{document}